\documentclass[twoside,bezier,reqno]{amsart}

\usepackage{graphicx}
\usepackage{mathptmx}
\usepackage{amsmath}
\usepackage{amssymb}
\usepackage{enumitem}
\usepackage{xcolor}
\usepackage{dsfont}
\numberwithin{equation}{section}
\newtheorem{theorem}{Theorem}

\newtheorem{corollary}[theorem]{Corollary}

\newtheorem{definition}[theorem]{Definition}
\newtheorem{remark}{Remark}

\begin{document}

\title{On Certain Recurrence Relations for Generalized Poly-Cauchy Numbers and Polynomials}

\author{Ghania Guettai}
\address[G. Guettai]{LaMA Laboratory, Dr Yahia Far\`{e}s University of M\'{e}d\'{e}a 26000 M\'{e}d\'{e}a, Algeria}
\email{guettai78@yahoo.fr}
\author{Diffalah Laissaoui}
\address[D. Laissaoui]{LaMA Laboratory, Dr Yahia Far\`{e}s University of M\'{e}d\'{e}a 26000 M\'{e}d\'{e}a, Algeria}
\email{laissaoui.diffalah74@gmail.com}
\author{Mohamed Amine Boutiche}
\address[M. A. Boutiche]{USTHB, Faculty of Mathematics, P.O. Box 32, El Alia 16111, Algiers, Algeria}
\email{boutiche@gmail.com}

\author{Mourad Rahmani}
\address[M. Rahmani]{LaMA Laboratory, USTHB, Faculty of Mathematics, P.O. Box 32, El Alia 16111, Algiers, Algeria}
\email{mourad.rahmani@gmail.com, mrahmani@usthb.dz}

\begin{abstract}
The main objective of this paper is to present recurrence relations for the generalized poly-Cauchy numbers and polynomials. This is accomplished by introducing the concept of generalized m-poly-Cauchy numbers and polynomials. Additionally, the paper delves into the discussion of the corresponding generalized m-poly-Bernoulli numbers and polynomials that are associated with the aforementioned generalized m-poly-Cauchy numbers and polynomials.
\end{abstract}
\maketitle

\section{Introduction}
Let $n\geq0,k\geq1$ be integers and let $a,q$ and $l_{1},\ldots,l_{k}$ be
non-zero real numbers, with $L=\left(  l_{1},\ldots,l_{k}\right)  $ and $l=%
{\textstyle\prod_{i=1}^{k}}l_{i}$. The Cauchy numbers of the first kind, denoted as $c_{n}$, and the second kind, represented as $\widehat{c}_{n}$, are defined in \cite{Comtet,Merlini} by the following formulas:
\[
c_{n}=\int_{0}^{1}x\left(  x-1\right)  \cdots\left(  x-n+1\right)  dx
\]
and
\begin{align*}
\widehat{c}_{n}  &  =\left(  -1\right)  ^{n}\int_{0}^{1}x\left(  x+1\right)
\cdots\left(  x+n-1\right)  dx.
\end{align*}

The generalized poly-Cauchy numbers of the first kind $c_{n}^{\left(
	k\right)  }\left(  a,q,L\right)  $ are defined in \cite{Komatsu1} by
\begin{multline}
c_{n}^{\left(  k\right)  }\left(  a,q,L\right)  =\int_{0}^{l_{1}}\cdots
\int_{0}^{l_{k}}\left(  x_{1}\cdots x_{k}\right)  ^{a}\left(  x_{1}\cdots
x_{k}-q\right) \\
\cdots\left(  x_{1}\cdots x_{k}-\left(  n-1\right)  q\right)  dx_{1}\cdots
dx_{k},\label{A}
\end{multline}
or by the following generating function
\[
l^{a}\operatorname*{Lif}\nolimits_{k}\left(  \frac{l\ln\left(  1+qx\right)
}{q};a\right)  =
{\displaystyle\sum\limits_{n\geq0}}
c_{n}^{\left(  k\right)  }\left(  a,q,L\right)  \frac{x^{n}}{n!},
\]
where $\operatorname*{Lif}\nolimits_{k}\left(  z;a\right)  $ denotes the
extended polylogarithm factorial function
\begin{equation}
\operatorname*{Lif}\nolimits_{k}\left(  z;a\right)  =
{\displaystyle\sum\limits_{n\geq0}}
\frac{1}{\left(  n+a\right)  ^{k}}\frac{z^{n}}{n!}. \label{K1}%
\end{equation}
 We can express
  $c_{n}^{\left(  k\right)  }\left(  a,q,L\right)$ in the following explicit expression:
\begin{equation}
c_{n}^{\left(  k\right)  }\left(  a,q,L\right)  =
{\displaystyle\sum\limits_{i=0}^{n}}
s\left(  n,i\right)  \frac{q^{n-i}l^{a+i}}{\left(  a+i\right)  ^{k}},
\label{F1}%
\end{equation}
in which $s\left(  n,i\right)  $ are the Stirling numbers of the
first kind, defined as follows:
\begin{equation}
\frac{1}{i!}\left(  \ln\left(  1+x\right)  \right)  ^{i}=
{\displaystyle\sum\limits_{n\geq i}}
s\left(  n,i\right)  \frac{x^{n}}{n!}, \label{F2}
\end{equation}
and given recursively by the formula
\[
s\left(  n+1,i\right)  =s\left(  n,i-1\right)  -ns\left(  n,i\right)  \text{
	\ }\left(  1\leq i\leq n\right)  .
\]
Note that, if $a=l=1,$ then $c_{n}^{\left(  k\right)  }\left(  1,q,\left(
1,\cdots,1\right)  \right)  :=c_{n,q}^{\left(  k\right)  }$ are the poly-Cauchy
numbers with a $q$ parameter introduced in \cite{Komatsu3} and, if we set $q=l=1$ in
(\ref{F1}) then, $c_{n}^{\left(  k\right)  }\left(  a,1,\left(  1,\cdots
,1\right)  \right)  :=c_{n}^{\left(  k\right)  }\left(  a\right)  $ are the
shifted poly-Cauchy numbers \cite{Komatsu2}.

Similar to $c_{n}^{\left(  k\right)  }\left(  a,q,L\right)$, the generalized
poly-Cauchy numbers of the second kind $\widehat{c}_{n}^{\left(  k\right)
}\left(  a,q,L\right)  $ are described in \cite{Komatsu1} by
\begin{multline}
\widehat{c}_{n}^{\left(  k\right)  }\left(  a,q,L\right)  =\left(  -1\right)
^{a-1}\int_{0}^{l_{1}}\cdots\int_{0}^{l_{k}}\left(  -x_{1}\cdots x_{k}\right)
^{a}\left(  -x_{1}\cdots x_{k}-q\right) \\
\cdots\left(  -x_{1}\cdots x_{k}-\left(  n-1\right)  q\right)  dx_{1}\cdots
dx_{k},\label{B}
\end{multline}
or explicitly as
\[
\widehat{c}_{n}^{\left(  k\right)  }\left(  a,q,L\right)  =
{\displaystyle\sum\limits_{i=0}^{n}}
s\left(  n,i\right)\left(  -1\right)  ^{i}  \frac{q^{n-i}l^{a+i}}{\left(
	a+i\right)  ^{k}}
\]
and have the generating function
\[
l^{a}\operatorname*{Lif}\nolimits_{k}\left(  -\frac{l\ln\left(  1+qx\right)
}{q};a\right)  =
{\displaystyle\sum\limits_{n\geq0}}
\widehat{c}_{n}^{\left(  k\right)  }\left(  a,q,L\right)  \frac{x^{n}}{n!}.
\]

In our research, our objective is to complete the study conducted in \cite{Komatsu1} by providing a straightforward method for calculating the integrals \eqref{A} and \eqref{B}. Subsequently, we extend all the results from \cite{Komatsu1} to the polynomial case. To achieve this, we introduce the concept of generalized $m$-poly-Cauchy numbers of the first and second kind.

%%In our research, we aim to introduce other types of generalized poly-Cauchy numbers and polynomials so as to get several recurrence relations and identities. The idea is to establish an infinite  matrix $\left(
%%\mathcal{C}^{(k)}_{n,m}\right)  _{n,m\geq0}$ that offers the generalized poly-Cauchy numbers in the first column.

We recall that, the combinatorial numbers which complement the $s\left(n,i\right)$ are the Stirling numbers of the second kind  $
\genfrac{\{}{\}}{0pt}{}{n}{i},
$ while  $\genfrac{\{}{\}}{0pt}{}{n}{i}_{r}$  introduced by Broder \cite{Broder} is an extension of $
\genfrac{\{}{\}}{0pt}{}{n}{i},$ and can be expressed by the exponential generating function
\begin{equation}
{\displaystyle\sum\limits_{n\geq i}}
\genfrac{\{}{\}}{0pt}{}{n+r}{i+r}_{r}
\frac{z^{n}}{n!}=\frac{1}{i!}e^{rz}\left(  e^{z}-1\right)  ^{i}. \label{GSti}
\end{equation}
Taking $r=0$ in \eqref{GSti}, we get $
\genfrac{\{}{\}}{0pt}{}{n}{i}.$ The properties
\begin{equation}
\genfrac{\{}{\}}{0pt}{0}{n}{r}_{r}=r^{n-r} \label{F4}
\end{equation}
and
\begin{equation}
\genfrac{\{}{\}}{0pt}{0}{n+r}{i+r}_{r}=\genfrac{\{}{\}}{0pt}{0}{n+r}{i+r}
_{r-1}-(r-1)
\genfrac{\{}{\}}{0pt}{0}{n+r-1}{i+r}
_{r-1}\label{F3}
\end{equation}
are given in \cite{Broder}.

\section{The generalized $m$-poly-Cauchy numbers}
We first introduce the generalized $m$-poly-Cauchy numbers of both kinds, then we give generating functions and recurrence formulas. Also, we prove some relations between the aforementioned two kinds of the generalized $m$-poly-Cauchy numbers.
\subsection{The generalized $m$-poly-Cauchy numbers of the first kind}
Let $n\geq0,m\geq0$ and $k\geq1$ be integers, and let $a,q$ and $l_{1}
,\ldots,l_{k}$ be non-zero real numbers, with $L=\left(  l_{1},\ldots
,l_{k}\right)  $ and $l=%
{\textstyle\prod_{i=1}^{k}}
l_{i}$. We give a sequence $\mathcal{C}_{n,m}^{\left(  k\right)  }\left(
a,q,L\right)  $ with two indices, which we call generalized $m$-poly-Cauchy numbers of the first kind, by

\begin{multline*}\label{A1}
\mathcal{C}_{n,m}^{\left(  k\right)  }\left(  a,q,L\right)  =\frac{\left(
		a+m\right)  ^{k}}{a^{k}}\int_{0}^{l_{1}}\cdots
\int_{0}^{l_{k}}\left(  x_{1}\cdots x_{k}\right)  ^{a+m}\left(  x_{1}\cdots
x_{k}-q\right) \\
\cdots\left(  x_{1}\cdots x_{k}-\left(  n-1\right)  q\right)  dx_{1}\cdots
dx_{k},
\end{multline*}

or simply by
	\[
	\mathcal{C}_{n,m}^{\left(  k\right)  }\left(  a,q,L\right)  =\frac{\left(
		a+m\right)  ^{k}}{a^{k}}%
		c_{n}^{\left(  k\right)  }\left(  a+m,q,L\right)
	\]

	Using \eqref{F1}, the generalized $m$-poly-Cauchy numbers of the first kind $\mathcal{C}_{n,m}^{\left(  k\right)  }\left(  a,q,L\right) $ may be expressed as%
	\[
	\mathcal{C}_{n,m}^{\left(  k\right)  }\left(  a,q,L\right)  =\frac{\left(
		a+m\right)  ^{k}}{a^{k}}%
		{\displaystyle\sum\limits_{i=0}^{n}}
	s\left(  n,i\right)  \frac{q^{n-i}l^{a+i}}{\left(  a+i+m\right)  ^{k}},\label{GBF1}
	\]
	with
	\[
	\mathcal{C}_{0,m}^{\left(  k\right)  }\left(  a,q,L\right)  =\frac{l^{a}%
	}{a^{k}}.
	\]

\begin{theorem}\label{N}
The expression for the generalized $m$-poly-Cauchy numbers of the first kind $\mathcal{C}_{n,m}^{\left( k\right) }\left( a,q,L\right) $ is given by
\begin{equation}
\mathcal{C}_{n,m}^{\left(  k\right)  }\left(  a,q,L\right)  =\frac{\left(
	a+m\right)  ^{k}}{l^{m}a^{k}}%
{\displaystyle\sum\limits_{j=0}^{m}}
\genfrac{\{}{\}}{0pt}{}{m+n}{j+n}%
_{n}q^{m-j}c_{n+j}^{\left(  k\right)  }\left(  a,q,L\right)  , \label{Q1}%
\end{equation}
with
\[
\mathcal{C}_{n,0}^{\left(  k\right)  }\left(  a,q,L\right)  =c_{n}^{\left(
	k\right)  }\left(  a,q,L\right)  .\text{\ }%
\]
\end{theorem}
\begin{proof}
	Using the inverse Stirling transform for (\ref{F1}), we have
	\[
	\frac{l^{a+n}}{q^{n}\left(  a+n\right)  ^{k}}=%
	{\displaystyle\sum\limits_{i=0}^{n}}
		\genfrac{\{}{\}}{0pt}{}{n}{i}
	q^{-i}c_{i}^{\left(  k\right)  }\left(  a,q,L\right)  .
	\]
	Now, from result in \cite[p. $681$, Corollary $1$]{Rahmani2014}, we deduce the formula in Theorem \ref{N}.
\end{proof}
\begin{theorem}
	 The $\mathcal{C}_{n,m}^{\left(  k\right)  }\left(  a,q,L\right) $ is given by
	\begin{equation}
	\frac{l^{a}\left(  a+m\right)  ^{k}}{a^{k}}\operatorname*{Lif}\nolimits_{k}%
	\left(  \frac{l\ln\left(  1+qx\right)  }{q};a+m\right)  =%
	{\displaystyle\sum\limits_{n\geq0}}
	\mathcal{C}_{n,m}^{\left(  k\right)  }\left(  a,q,L\right)  \frac{x^{n}}%
	{n!}.\label{GEN1KIND}%
	\end{equation}
\end{theorem}
\begin{proof}
	We find, from(\ref{F2}) and (\ref{K1}), that
	\begin{align*}%
	{\displaystyle\sum\limits_{n\geq0}}
	\mathcal{C}_{n,m}^{\left(  k\right)  }\left(  a,q,L\right)  \frac{x^{n}}{n!}
	&  =\frac{l^{a}\left(  a+m\right)  ^{k}}{a^{k}}%
	{\displaystyle\sum\limits_{i\geq0}}
	\left(  \frac{l}{q}\right)  ^{i}\frac{1}{\left(  a+i+m\right)  ^{k}}%
		{\displaystyle\sum\limits_{n\geq i}}
	s\left(  n,i\right)  \frac{\left(  qx\right)  ^{n}}{n!}\\
	&  =\frac{l^{a}\left(  a+m\right)  ^{k}}{a^{k}}%
	{\displaystyle\sum\limits_{i\geq0}}
	\frac{1}{\left(  a+i+m\right)  ^{k}}\left(  \frac{l\ln\left(  1+qx\right)
	}{q}\right)  ^{i}\frac{1}{i!}\\
	&  =\frac{l^{a}\left(  a+m\right)  ^{k}}{a^{k}}\operatorname*{Lif}%
	\nolimits_{k}\left(  \frac{l\ln\left(  1+qx\right)  }{q};a+m\right)  .
	\end{align*}
	
\end{proof}
\begin{theorem}
	\label{AA}The $\mathcal{C}_{n,m}^{\left(  k\right)  }\left(  a,q,L\right)  $
	satisfies the following recurrence relation
	\begin{equation}
	\mathcal{C}_{n+1,m}^{\left(  k\right)  }\left(  a,q,L\right)  =l\left(
	1-\frac{1}{a+m+1}\right)  ^{k}\mathcal{C}_{n,m+1}^{\left(  k\right)  }\left(
	a,q,L\right)  -nq\mathcal{C}_{n,m}^{\left(  k\right)  }\left(  a,q,L\right)
	,\label{Re1}%
	\end{equation}
	with
	\[
	\mathcal{C}_{0,m}^{\left(  k\right)  }\left(  a,q,L\right)  =\frac{l^{a}%
	}{a^{k}}.
	\]
\end{theorem}
\begin{proof}
	According to (\ref{Q1}) and (\ref{F3}), we can write
	\begin{multline*}
	\mathcal{C}_{n+1,m}^{\left(  k\right)  }\left(  a,q,L\right)  =\frac
	{q^{m}\left(  a+m\right)  ^{k}}{l^{m}a^{k}}%
	{\displaystyle\sum\limits_{j=0}^{m}}
	q^{-j}%
		\genfrac{\{}{\}}{0pt}{}{m+n+1}{j+n+1}%
	_{n}c_{n+j+1}^{\left(  k\right)  }\left(  a,q,L\right)  \\
	-\frac{q^{m}\left(  a+m\right)  ^{k}}{l^{m}a^{k}}%
	{\displaystyle\sum\limits_{j=0}^{m}}
	q^{-j}n%
	\genfrac{\{}{\}}{0pt}{}{m+n}{j+n+1}%
		_{n}c_{n+j+1}^{\left(  k\right)  }\left(  a,q,L\right)  .
	\end{multline*}
	Using (\ref{F4}), we get
	\begin{multline*}
	\mathcal{C}_{n+1,m}^{\left(  k\right)  }\left(  a,q,L\right)  =\frac
	{q^{m+1}\left(  a+m\right)  ^{k}}{l^{m}a^{k}}%
	{\displaystyle\sum\limits_{j=1}^{m+1}}
	q^{-j}%
		\genfrac{\{}{\}}{0pt}{}{m+n+1}{j+n}%
	_{n}c_{n+j}^{\left(  k\right)  }\left(  a,q,L\right)  \\
	-\frac{q^{m+1}\left(  a+m\right)  ^{k}}{l^{m}a^{k}}n%
	{\displaystyle\sum\limits_{j=1}^{m}}
	q^{-j}%
	\genfrac{\{}{\}}{0pt}{}{m+n}{j+n}%
	_{n}c_{n+j}^{\left(  k\right)  }\left(  a,q,L\right)  \\ +\frac{q^{m+1}\left(
		a+m\right)  ^{k}}{l^{m}a^{k}}n^{m+1}c_{n}^{\left(  k\right)  }\left(
	a,q,L\right)
	-n\frac{q^{m+1}\left(  a+m\right)  ^{k}}{l^{m}a^{k}}n^{m}c_{n}^{\left(
		k\right)  }\left(  a,q,L\right),
	\end{multline*}
	which gives the following equality
	\begin{multline*}
	\mathcal{C}_{n+1,m}^{\left(  k\right)  }\left(  a,q,L\right)  =\frac
	{q^{m+1}\left(  a+m\right)  ^{k}}{l^{m}a^{k}}{\displaystyle\sum\limits_{j=0}%
		^{m+1}}q^{-j}%
	\genfrac{\{}{\}}{0pt}{}{m+n+1}{j+n}%
	_{n}c_{n+j}^{\left(  k\right)  }(a,q,L)\\
	-nq\left(  \frac{q^{m}\left(  a+m\right)  ^{k}}{l^{m}a^{k}}{\displaystyle\sum
		\limits_{j=0}^{m}}q^{-j}%
	\genfrac{\{}{\}}{0pt}{}{m+n}{j+n}%
	_{n}c_{n+j}^{\left(  k\right)  }(a,q,L)\right)  .
	\end{multline*}
	This is evidently equivalent to (\ref{Re1}).
\end{proof}
 Consequently, from Theorem \ref{AA}, we can deduce a recurrence formula
 for the generalized $m$-poly-Cauchy numbers of the first kind with negative upper indices $\mathcal{C}_{n,m}^{\left(  -k\right)  }\left(  a,q,L\right)  $.
\begin{corollary}
	The $\mathcal{C}_{n,m}^{\left(  -k\right)  }\left(  a,q,L\right)  $ satisfies
	the recurrence equation
	\begin{equation}
	\mathcal{C}_{n+1,m}^{\left(  -k\right)  }\left(  a,q,L\right)  =l\left(
	1+\frac{1}{a+m}\right)  ^{k}\mathcal{C}_{n,m+1}^{\left(  -k\right)  }\left(
	a,q,L\right)  -nq\mathcal{C}_{n,m}^{\left(  -k\right)  }\left(  a,q,L\right)
	,\label{Re2}%
	\end{equation}
	with
	\[
	\mathcal{C}_{0,m}^{\left(  -k\right)  }\left(  a,q,L\right)  =a^{k}l^{a}.
	\]
\end{corollary}
\begin{remark}
	Note that the double generating function of $\ \mathcal{C}_{n,m}^{\left(		-k\right)  }\left(  a,q,L\right)  $ can be easily obtained by using
	(\ref{GEN1KIND}).
	\begin{align*}%
		{\displaystyle\sum\limits_{n,k\geq0}}
	\mathcal{C}_{n,m}^{\left(  -k\right)  }\left(  a,q,L\right)  \frac{x^{n}}%
	{n!}\frac{y^{k}}{k!}  & =%
	{\displaystyle\sum\limits_{k\geq0}}
		\frac{l^{a}a^{k}}{\left(  a+m\right)  ^{k}}%
		{\displaystyle\sum\limits_{p\geq0}}
	\frac{\left(  a+m+p\right)  ^{k}}{p!}\left(  \frac{l\ln\left(  1+qx\right)
	}{q}\right)  ^{p}\frac{y^{k}}{k!}\\
	& =l^{a}%
	{\displaystyle\sum\limits_{p\geq0}}
	\frac{1}{p!}\left(  \frac{l\ln\left(  1+qx\right)  }{q}\right)  ^{p}%
	{\displaystyle\sum\limits_{k\geq0}}
	\left(  \frac{a\left(  a+m+p\right)  }{a+m}y\right)  ^{k}\frac{1}{k!}\\
	& =l^{a}e^{ay}%
	{\displaystyle\sum\limits_{p\geq0}}
	\frac{1}{p!}\left(  \frac{l\ln\left(  1+qx\right)  }{q}e^{\frac{ay}{a+m}%
	}\right)  ^{p}\\
	& =l^{a}e^{ay}\exp\left(  \ln\left(  1+qx\right)  ^{\frac{l}{q}e^{\frac
			{ay}{a+m}}}\right)  \\
	& =l^{a}e^{ay}\left(  1+qx\right)  ^{\frac{l}{q}e^{\frac{ay}{a+m}}}.
	\end{align*}
\end{remark}
Now, if $a=l=q=k=1,$ then $\mathcal{C}_{n,m}^{\left(  1\right)  }\left(  1,1,\left(
1,\cdots,1\right)  \right)  :=\mathcal{C}_{n,m}$ is $m-$Cauchy numbers of the first kind which are different from $p$-Cauchy numbers of the first kind \cite{Rahmani2016}. The first few values for $\mathcal{C}_{n,m}$ are:%
\begin{align*}
\mathcal{C}_{0,m} &  =1,\\
\text{ \ }\mathcal{C}_{1,m} &  =\frac{1+m}{2+m},\text{ }\\
\text{\ }\mathcal{C}_{2,m} &  =-\frac{1+m}{\left( 2+ m\right)  \left(
3+m\right)  },\text{ }\\
\text{\ }\mathcal{C}_{3,m} &  =\frac{\left(  1+m\right)  \left(  +m\right)
}{\left(  2+m\right)  \left(  3+m\right)  \left( 4+ m\right)  },\\
\text{\ }\mathcal{C}_{4,m} &  =\frac{-2\left(  1+m\right)  \left(38+12m+
m^{2}\right)  }{\left( 2+m\right)  \left( 3+m\right)  \left(
4+m\right)  \left(  5+m\right)  }.
\end{align*}
\begin{corollary}
	The numbers $\mathcal{C}_{n,m}$ satisfy the
	recurrence relation
	\begin{equation}
	\left(2+m\right)\mathcal{C}_{n+1,m}=\left(1+m\right)\mathcal{C}_{n,m+1}-n\left(2+m\right)\mathcal{C}%
	_{n,m},\label{Re3}%
	\end{equation}
	with $\mathcal{C}_{0,m}=1$ is the initial sequence.
\end{corollary}
\[
(\mathcal{C}_{n,m})_{n,m \geq 0}=%
\begin{pmatrix}
1 & 1 & 1 & 1 & \cdots  \\
1/2 & 2/3 & 3/4 & 4/5 & \cdots \\
-1/6 & -1/6 & -3/20 & -2/15 & \cdots \\
1/4 & 7/30 & 1/5 & 6/35 & \cdots \\
-19/30 & -17/30 & -33/70 & -83/210 & \cdots \\
%52 & 167/12 & 127/21 & 185/56 & \cdots & \mathcal{C}_{5,m}\\
%203 & 2057/42 & 235/12 & 8389/840 & \cdots & \mathcal{C}_{6,m}\\
\vdots & \vdots & \vdots & \vdots &  & \\
\mathcal{C}_{n,0} & \mathcal{C}_{n,1} & \mathcal{C}_{n,2} & \mathcal{C}_{n,3}
&  &
\end{pmatrix}
\]
The Gregory coefficients $\mathcal{G}_{n}$ are defined by%
\[
\mathcal{G}_{n}=\frac{\mathcal{C}_{n}}{n!},
\]
and play an important role in Gregory's formula \cite{Merlini} which has similarities to the
Euler-Maclaurin summation formula wherein finite differences replace
derivatives. Using (\ref{Re3}), we can easily prove the following recurrence
formula for computing $\mathcal{G}_{n}$.
\begin{corollary}
	
	Let an initial sequence $\mathcal{G}_{0,m}=1,$ be given. Then the matrix $\left(  \mathcal{G}_{n,m}\right)
	_{n,m\geq0}$ associated with this sequence is given recursively by the formula
	\[
	\left(  1+n\right)  \left(  2+m\right)\mathcal{G}_{n+1,m}=\left(1+m\right)\mathcal{G}_{n,m+1}-n\left(2+m\right)\mathcal{G}_{n,m}
	\]
	in which the first column of $\left(  \mathcal{G}_{n,m}\right)
	_{n,m\geq0}$ is 	$\mathcal{G}_{n,0}=\mathcal{G}_{n}.$
\end{corollary}

\subsection{The generalized $m$-poly-Cauchy numbers of the second kind}
Similarly, we define the generalized $m$-poly-Cauchy  numbers of the second
kind $\widehat{\mathcal{C}}_{n,m}^{\left(  k\right)  }\left(  a,q,L\right)  $
by%
\[
	\widehat{\mathcal{C}}_{n,m}^{\left(  k\right)  }\left(  a,q,L\right)  =\frac{\left(
		a+m\right)  ^{k}}{a^{k}}%
		\widehat{c}_{n}^{\left(  k\right)  }\left(  a+m,q,L\right).
	\]

\begin{theorem}
	\begin{enumerate}
\item[\ ]
		\item The exponential generating function for $\widehat{\mathcal{C}}_{n,m}^{\left(  k\right)  }\left(  a,q,L\right)$ is
		\[
		\frac{l^{a}\left(  a+m\right)  ^{k}}{a^{k}}\operatorname*{Lif}\nolimits_{k}%
		\left(  -\frac{l\ln\left(  1+qx\right)  }{q};a+m\right)  =%
		{\displaystyle\sum\limits_{n\geq0}}
		\widehat{\mathcal{C}}_{n,m}^{\left(  k\right)}\left(  a,q,L\right)  \frac{x^{n}}{n!}.
		\]	
		\item The $\widehat{\mathcal{C}}_{n,m}^{\left(  k\right)  }\left(  a,q,L\right)$ may be expressed as
		\[
		\widehat{\mathcal{C}}_{n,m}^{\left(  k\right)  }\left(  a,q,L\right)
		=\frac{\left(  a+m\right)  ^{k}}{a^{k}}%
			{\displaystyle\sum\limits_{i=0}^{n}}
			q^{n-i}s\left(  n,i\right)  \frac{\left(  -1\right)  ^{i}l^{a+i}}{\left(
			a+i+m\right)  ^{k}}.
		\]	

\item The expression for the generalized $m$-poly-Cauchy numbers of the second kind $\widehat{\mathcal{C}}_{n,m}^{\left( k\right) }\left( a,q,L\right) $ is given by
 \[
\widehat{\mathcal{C}}_{n,m}^{\left(  k\right)  }\left(  a,q,L\right)
=\frac{\left(  -q\right)  ^{m}\left(  a+m\right)  ^{k}}{l^{m}a^{k}}%
{\displaystyle\sum\limits_{j=0}^{m}}
q^{-j}%
\genfrac{\{}{\}}{0pt}{}{m+n}{j+n}%
_{n}\widehat{c}_{n+j}^{\left(  k\right)  }\left(  a,q,L\right).
\]
		\item The $\widehat{\mathcal{C}}_{n,m}^{\left(  k\right)  }\left(
		a,q,L\right)  $ has the  recurrence formula
		\[
		\widehat{\mathcal{C}}_{n+1,m}^{\left(  k\right)  }\left(  a,q,L\right)
		=-l\left(  1-\frac{1}{a+m+1}\right)  ^{k}\widehat{\mathcal{C}}_{n,m+1}%
		^{\left(  k\right)  }\left(  a,q,L\right)  -nq\widehat{\mathcal{C}}%
		_{n,m}^{\left(  k\right)  }\left(  a,q,L\right)  ,
		\]
		with %
		\[
		\widehat{\mathcal{C}}_{0,m}^{\left(  k\right)  }\left(  a,q,L\right)
		=\frac{l^{a}}{a^{k}}.
		\]
		
		\item The $\widehat{\mathcal{C}}_{n,m}^{\left(  -k\right)  }\left(
		a,q,L\right)  $ has the recurrence formula
		\[
		\widehat{\mathcal{C}}_{n+1,m}^{\left(  -k\right)  }\left(  a,q,L\right)
		=-l\left(  1+\frac{1}{a+m}\right)  ^{k}\widehat{\mathcal{C}}_{n,m+1}^{\left(
			-k\right)  }\left(  a,q,L\right)  -nq\widehat{\mathcal{C}}_{n,m}^{\left(
			-k\right)  }\left(  a,q,L\right)  ,
		\]
		with
		\[
		\widehat{\mathcal{C}}_{0,m}^{\left(  -k\right)  }\left(  a,q,L\right)
		=a^{k}l^{a}.
		\]
		\item The double generating function of $\ \widehat{\mathcal{C}}%
		_{n,m}^{\left(  -k\right)  }\left(  a,q,L\right)  $ is given by%
		
		\[
		{\displaystyle\sum\limits_{n,k\geq0}}
		\widehat{\mathcal{C}}_{n,m}^{\left(  -k\right)  }\left(  a,q,L\right)
		\frac{x^{n}}{n!}\frac{y^{k}}{k!}=\frac{l^{a}e^{ay}}{\left(  1+qx\right)
			^{\frac{l}{q}e^{\frac{ay}{a+m}}}}.
		\]
		\end{enumerate}
\end{theorem}
Next, we shall prove some relations between the aforementioned two kinds of the generalized
$m$-poly-Cauchy numbers.
\begin{theorem}\label{aze}
	For any integer $n\geq1,$
	\begin{align}
	\left(  -1\right)  ^{n}\frac{\mathcal{C}_{n,m}^{\left(  k\right)  }\left(
		a,q,L\right)  }{n!}  &  =%
		{\displaystyle\sum\limits_{i=1}^{n}}
	q^{n-i}\binom{n-1}{i-1}\frac{\widehat{\mathcal{C}}_{i,m}^{\left(  k\right)
		}\left(  a,q,L\right)  }{i!}.\label{AN1}\\
	\left(  -1\right)  ^{n}\frac{\widehat{\mathcal{C}}_{n,m}^{\left(  k\right)
		}\left(  a,q,L\right)  }{n!}  &  =%
	{\displaystyle\sum\limits_{i=1}^{n}}
	q^{n-i}\binom{n-1}{i-1}\frac{\mathcal{C}_{i,m}^{\left(  k\right)  }\left(
		a,q,L\right)  }{i!}. \label{AN2}%
	\end{align}
\end{theorem}
\begin{proof}
	We use generating function formula, yields
	\[%
		{\displaystyle\sum\limits_{n\geq0}}
	\left(  -1\right)  ^{n}\frac{\mathcal{C}_{n,m}^{\left(  k\right)  }\left(
		a,q,L\right)  }{n!}x^{n}=\frac{l^{a}\left(  a+m\right)  ^{k}}{a^{k}%
	}\operatorname*{Lif}\nolimits_{k}\left(  \frac{l\ln\left(  1-qx\right)  }
	{q};a+m\right).
	\]
	On the other hand,
	\[
		{\displaystyle\sum\limits_{n\geq0}}
		{\displaystyle\sum\limits_{i=1}^{n}}
	\binom{n-1}{i-1}q^{n-i}\frac{\widehat{\mathcal{C}}_{i,m}^{\left(  k\right)
		}\left(  a,q,L\right)  }{i!}x^{n}=%
		{\displaystyle\sum\limits_{i\geq1}}
	\frac{\widehat{\mathcal{C}}_{i,m}^{\left(  k\right)  }\left(  a,q,L\right)
	}{i!q^{i}}%
	{\displaystyle\sum\limits_{n\geq i}}
		\binom{n-1}{i-1}\left(  qx\right)  ^{n}.
	\]
	Since
	\[%
		{\displaystyle\sum\limits_{n\geq i}}
	\binom{n-1}{i-1}z^{n}=\frac{z^{i}}{\left(  1-z\right)  ^{i}},
	\]
	we get%
	\begin{align*}%
	{\displaystyle\sum\limits_{n\geq0}}
	{\displaystyle\sum\limits_{i=1}^{n}}
	%EndExpansion
	\binom{n-1}{i-1}q^{n-i}\frac{\widehat{\mathcal{C}}_{i,m}^{\left(  k\right)
		}\left(  a,q,L\right)  }{i!}x^{n}  &  =%
	{\displaystyle\sum\limits_{i\geq1}}
		\frac{\widehat{\mathcal{C}}_{i,m}^{\left(  k\right)  }\left(  a,q,L\right)
	}{i!}\frac{x^{i}}{\left(  1-qx\right)  ^{i}}\\
	&  =\frac{l^{a}\left(  a+m\right)  ^{k}}{a^{k}}\operatorname*{Lif}
	\nolimits_{k}\left(  \frac{l\ln\left(  1-qx\right)  }{q};a+m\right).
	\end{align*}
	Thus, the proof of (\ref{AN1}) is completed. Using the same method we can prove
	(\ref{AN2}).
\end{proof}
\begin{remark}
	If $m=0$ and $q=1$, Theorem \ref{aze} becomes
	Theorem $13$ in \cite{Komatsu1}.
\end{remark}
%%%%%%%%%%%%%%%%%%%%%%%%%%%%%%%%%%%%%%%%%%%%%%%%%%%%%%%%%%%%%%%%%%%%%%%%%%%%%%%%%%%%%%%%%%%%%%%%%%%
%%%%%%%%ù    Bernoulli numbers
%%%%%%%%%%%%%%%%%%%%%%%%%%%%%%%%%%%%%%%%%%%%%%%%%%%%%%%%%%%%%%%%%%%%%%%%%%%%%%%%%%%%%%%%%%%%%%%%%%%%

\section{The generalized $m$-poly-Bernoulli numbers}
The generalized poly-Bernoulli numbers $B_{n}^{\left(  k\right)  }\left(
a,L\right)  $ are defined in \cite{Komatsu1} as follows:
\[
l^{a-1}\frac{\operatorname*{Li}_{k}\left(  l\left(  1-e^{-z}\right)
	;a-1\right)  }{\left(  1-e^{-z}\right)  }=%
{\displaystyle\sum\limits_{n\geq0}}
B_{n}^{\left(  k\right)  }\left(  a,L\right)  \frac{z^{n}}{n!},
\]
wherein $\operatorname*{Li}_{k}\left(  z;a\right)  $ denotes the generalized
polylogarithm function%
\[
\operatorname*{Li}{}_{k}\left(  z;a\right)  =%
{\displaystyle\sum\limits_{n\geq1}}
\frac{z^{n}}{\left(  n+a\right)  ^{k}}.
\]
The $B_{n}^{\left(  k\right)  }\left(  a,L\right) $ can be computed explicitly by the formula
\[
B_{n}^{\left(  k\right)  }\left(  a,L\right)  =%
{\displaystyle\sum\limits_{i=0}^{n}}
\left(  -1\right)  ^{n-i}%
\genfrac{\{}{\}}{0pt}{}{n}{i}%
\frac{i!l^{i+a}}{\left(  a+i\right)  ^{k}}.
\]
 Here we introduce the generalized $m$-poly-Bernoulli numbers, denoted $	\mathcal{B}_{n,m}^{\left(  k\right)  }\left(  a,q,L\right)$, corresponding to the generalized $m$-poly-Cauchy numbers and give the connections between them.
\begin{definition}
	Let $n\geq0,m\geq0$ and $k\geq1$ be integers, and let $a,q$ and $l_{1}%
	,\ldots,l_{k}$ be non-zero real numbers, with $L=\left(  l_{1},\ldots
	,l_{k}\right)  $ and $l=%
		{\textstyle\prod_{i=1}^{k}}
	l_{i}$. The $\mathcal{B}%
	_{n,m}^{\left(  k\right)  }\left(  a,q,L\right)  $ is defined by:
	\begin{equation}
	\frac{ql^{a-1}\left(  a+m\right)  ^{k}}{a^{k}}\frac{\operatorname*{Li}%
		_{k}\left(  \frac{l}{q}\left(  1-e^{-qz}\right)  ;a+m-1\right)  }{1-e^{-qz}}=%
		{\displaystyle\sum\limits_{n\geq0}}
	\mathcal{B}_{n,m}^{\left(  k\right)  }\left(  a,q,L\right)  \frac{z^{n}}{n!}.
	\label{GenFin}%
	\end{equation}
\end{definition}
\begin{theorem}\label{GenFi}
	An explicit formula for $\mathcal{B}_{n,m}^{\left(  k\right)  }\left(
	a,q,L\right)  $ is
	\[
	\mathcal{B}_{n,m}^{\left(  k\right)  }\left(  a,q,L\right)  =\frac{\left(
		a+m\right)  ^{k}}{a^{k}}%
		{\displaystyle\sum\limits_{i=0}^{n}}
		\frac{i!\left(  -q\right)  ^{n-i}l^{i+a}}{\left(  a+m+i\right)  ^{k}}%
		\genfrac{\{}{\}}{0pt}{}{n}{i}.
	\]
\end{theorem}
\begin{proof}
	Equations (\ref{GenFin}) and (\ref{GSti}) imply that
	\begin{align*}%
		{\displaystyle\sum\limits_{n\geq0}}
	\mathcal{B}_{n,m}^{\left(  k\right)  }\left(  a,q,L\right)  \frac{z^{n}}{n!}
	&  =\frac{ql^{a-1}\left(  a+m\right)  ^{k}}{a^{k}\left(  1-e^{-qz}\right)  }%
		{\displaystyle\sum\limits_{i\geq1}}
	\frac{\left(  \frac{l}{q}\left(  1-e^{-qz}\right)  \right)  ^{i}}{\left(
		i+a+m-1\right)  ^{k}}\\
	&  =\frac{l^{a}\left(  a+m\right)  ^{k}}{a^{k}}%
		{\displaystyle\sum\limits_{i\geq0}}
	\frac{\left(  \frac{l}{q}\left(  1-e^{-qz}\right)  \right)  ^{i}}{\left(
		i+a+m\right)  ^{k}}\\
	&  =\frac{\left(  a+m\right)  ^{k}}{a^{k}}%
	{\displaystyle\sum\limits_{i\geq0}}
	\frac{\left(  -1\right)  ^{i}i!l^{i+a}}{q^{i}\left(  i+a+m\right)  ^{k}}%
	{\displaystyle\sum\limits_{n\geq i}}
	\genfrac{\{}{\}}{0pt}{}{n}{i}%
	\frac{\left(  -qz\right)  ^{n}}{n!}\\
	&  =%
	{\displaystyle\sum\limits_{n\geq0}}
	\left(  \frac{\left(  a+m\right)  ^{k}}{a^{k}}%
		{\displaystyle\sum\limits_{i=0}^{n}}
	\frac{\left(  -q\right)  ^{n-i}i!l^{i+a}}{\left(  i+a+m\right)  ^{k}}%
	\genfrac{\{}{\}}{0pt}{}{n}{i}%
	\right)  \frac{z^{n}}{n!}.
	\end{align*}
	Comparing the first formula with the last one, we get Theorem \ref{GenFi}.
\end{proof}
Using the inverse Stirling transform \cite{Rahmani2014}, we give the following corollary.
\begin{corollary}
	One has%
	\[%
		{\displaystyle\sum\limits_{i=0}^{n}}
	s\left(  n,i\right)  \left(  -q\right)  ^{n-i}\mathcal{B}_{i,m}^{\left(
		k\right)  }\left(  a,q,L\right)  =\frac{(a+m)^{k}n!l^{n+a}}{\left(
		a  \left(  n+a+m\right)  \right)  ^{k}}.
	\]
\end{corollary}
In this paragraph, to find a simple method for computing the generalized $m$-poly-Bernoulli numbers, we consider the sequence $\left(
H_{n,m,p}^{\left(  k\right)  }\left(  a,q,L\right)  \right)  _{n,p\geq0}$ with
two indices defined by%
\begin{equation}
H_{n,p}\left(  m\right)  :=H_{n,m,p}^{\left(  k\right)  }\left(  a,q,L\right)
=\frac{1}{p!l^{p}}\left(  \frac{p+a+m}{a+m}\right)  ^{k}\sum_{i=0}^{p}s\left(
p,i\right)  \left(  -q\right)  ^{p-i}\mathcal{B}_{n+i,m}^{(k)}\left(
a,q,L\right),  \label{Gq1}%
\end{equation}
with%
\[
H_{0,p}\left(  m\right)  =\frac{l^{a}}{a^{k}}%
\]
and
\[
H_{n,0}\left(  m\right)  =\mathcal{B}_{n,m}^{(k)}\left(  a,q,L\right)  .
\]
\begin{theorem}
The $H_{n,p}\left(  m\right)  $ is given recursively by the formula
\begin{equation}
H_{n+1,p}\left(  m\right)  =\frac{l\left(  p+1\right)  \left(  p+a+m\right)
^{k}}{\left(  p+a+m+1\right)  ^{k}}H_{n,p+1}\left(  m\right)  -pqH_{n,p}%
\left(  m\right),  \label{TGQ1}%
\end{equation}
with
\[
H_{0,p}\left(  m\right)  =\frac{l^{a}}{a^{k}}.
\]
\end{theorem}
\begin{proof}
Taking equation ($\ref{Gq1})$ and using the recurrence formula for $s(n,i),$ we find that
\begin{align*}
H_{n,p+1}\left(  m\right)   &  =\frac{(p+a+m+1)^{k}}{\left(  p+1\right)
!l^{p+1}(a+m)^{k}}\sum_{i=0}^{p+1}s\left(  p+1,i\right)  \left(  -q\right)
^{p+1-i}\mathcal{B}_{n+i,m}^{(k)}\left(  a,q,L\right)  \\
&  =\frac{(p+a+m+1)^{k}}{\left(  p+1\right)  !l^{p+1}(a+m)^{k}}\sum
_{i=0}^{p+1}\left(  s\left(  p,i-1\right)  -ps\left(  p,i\right)  \right)
\left(  -q\right)  ^{p+1-i}\mathcal{B}_{n+i,m}^{(k)}\left(  a,q,L\right)  \\
&  =\frac{\left(  p+a+m+1\right)  ^{k}}{\left(  p+a+m\right)  ^{k}\left(
p+1\right)  l}\left(  H_{n+1,p}\left(  m\right)  +qpH_{n,p}\left(  m\right)
\right)  ,
\end{align*}
which is obviously equivalent to $\left(  \ref{TGQ1}\right)  $.
\end{proof}
Now, setting $m=0$ and letting $p:=m$, in (\ref{TGQ1}), we get the following recurrence equation for calculating $\mathcal{H}_{n,0}^{(k)}\left(  a,q,L\right):={B}_{n}^{(k)}\left(  a,q,L\right)$.
\begin{corollary} We have
\begin{equation}
\mathcal{H}_{n+1,m}^{(k)}\left(  a,q,L\right) =\frac{l\left(  m+1\right)  \left(  m+a\right)
^{k}}{\left(  m+a+1\right)  ^{k}}\mathcal{H}_{n,m+1}^{(k)}\left(  a,q,L\right)   -mq\mathcal{H}_{n,m}^{(k)}\left(  a,q,L\right),\label{Alg1}
\end{equation}
with condition
\[
\mathcal{H}_{0,m}^{(k)}\left(  a,q,L\right)  :=\frac{l^{a}}{a^{k}}.
\]
\end{corollary}
\begin{remark}
By setting $k=1,a=1,q=1$ and $l=1$ in (\ref{Alg1}), we get Rahmani's algorithm \cite{Rahmani2015} for computing the classical Bernoulli numbers with $\mathcal{B}_{1}=\frac{1}{2}$.
\end{remark}
By replacing $k$ by $-k$ in the previous theorem, we can easily obtain the
values of the $\mathcal{B}_{n,m}^{(k)}\left(  a,q,L\right)  $ with negative
upper indices.
\begin{corollary} We have
\begin{equation}
H_{n+1,m,p}^{\left(  -k\right)  }\left(  a,q,L\right) =\frac{l\left(  p+1\right)\left(  p+a+m+1\right)  ^{k}}{\left(  p+a+m\right)
^{k}}H_{n,m+1,p}^{\left(  -k\right)  }\left(  a,q,L\right)  -pqH_{n,m,p}^{\left( -k\right)  }\left(  a,q,L\right),  \label{TGQ11}%
\end{equation}
with conditions
\[
H_{0,m,p}^{\left(  -k\right)  }\left(  a,q,L\right)  ={l^{a}{a^{k}}}
\]
and
\[
H_{n,m,0}^{\left(  -k\right)  }\left(  a,q,L\right)  =\mathcal{B}_{n,m}^{(-k)}\left(  a,q,L\right).
\]
\end{corollary}
\begin{theorem}
The double exponential generating function of $\mathcal{B}_{n,m}^{(-k)}\left(
a,q,L\right)  $ is
\[
\sum_{n\geq0}\sum_{k\geq0}\mathcal{B}_{n,m}^{(-k)}\left(  a,q,L\right)
\frac{x^{n}}{n!}\frac{y^{k}}{k!}=\frac{l^{a}e^{ay}}{1-\frac{l}{q}\left(
1-e^{-qz}\right)  e^{\frac{ay}{a+m}}}.
\]
\end{theorem}
\begin{proof}
From (\ref{GenFin}), we have%
\begin{align*}
\sum_{n\geq0}\sum_{k\geq0}\mathcal{B}_{n,m}^{(-k)}\left(  a,q,L\right)
\frac{x^{n}}{n!}\frac{y^{k}}{k!}  & =\sum_{k\geq0}\left(  \frac{ql^{a-1}a^{k}%
}{\left(  a+m\right)  ^{k}}\sum_{p\geq0}\frac{\left(  \frac{l}{q}\left(
1-e^{-qz}\right)  \right)  ^{p}\left(  a+m+p-1\right)  ^{k}}{\left(
1-e^{-qz}\right)  }\right)  \frac{y^{k}}{k!}\\
& =l^{a}\sum_{p\geq1}\left(  \frac{l}{q}\left(  1-e^{-qz}\right)  \right)
^{p-1}\sum_{k\geq0}\left(  \frac{a\left(  a+m+p-1\right)  y}{\left(
a+m\right)  }\right)  ^{k}\frac{1}{k!}\\
& =l^{a}e^{ay}\sum_{p\geq1}\left(  \frac{l}{q}\left(  1-e^{-qz}\right)
e^{\frac{ay}{a+m}}\right)  ^{p-1}\\
& =\frac{l^{a}e^{ay}}{1-\frac{l}{q}\left(  1-e^{-qz}\right)  e^{\frac{ay}%
{a+m}}}.
\end{align*}
\end{proof}
\begin{theorem}
	 The $\mathcal{B}_{n,m}^{\left(  k\right)  }\left(
	a,q,L\right)  $ is related to $\mathcal{C}_{n,m}^{\left(  k\right)  }\left(  a,q,L\right)  $ by the following explicit formula:
	\[
	\mathcal{B}_{n,m}^{\left(  k\right)  }\left(  a,q,L\right)  =%
	{\displaystyle\sum\limits_{j=1}^{n}}
	q^{n-j}%
	{\displaystyle\sum\limits_{i=1}^{n}}
	i!%
	\genfrac{\{}{\}}{0pt}{}{n}{i}%
	\genfrac{\{}{\}}{0pt}{}{i-1}{j-1}%
	\mathcal{C}_{j,m}^{\left(  k\right)  }\left(  a,q,L\right), \ \ \ (n\geq 1).
	\]
\end{theorem}
\begin{proof}
	We can write
	\begin{align*}
	RHS  & =%
	{\displaystyle\sum\limits_{j=1}^{n}}
	q^{n-j}%
	{\displaystyle\sum\limits_{i=1}^{n}}
	i!%
	\genfrac{\{}{\}}{0pt}{}{n}{i}%
	\genfrac{\{}{\}}{0pt}{}{i-1}{j-1}%
	\frac{\left(  a+m\right)  ^{k}}{a^{k}}%
	{\displaystyle\sum\limits_{p=0}^{j}}
	q^{j-p}s\left(  j,p\right)  \frac{l^{a+p}}{\left(  a+p+m\right)  ^{k}}\\
	& =\frac{\left(  a+m\right)  ^{k}}{a^{k}}%
	{\displaystyle\sum\limits_{p=0}^{n}}
	q^{n-p}\frac{l^{a+p}}{\left(  a+p+m\right)  ^{k}}\left(
	{\displaystyle\sum\limits_{i=1}^{n}}
	i!%
	\genfrac{\{}{\}}{0pt}{}{n}{i}%
	{\displaystyle\sum\limits_{j=1}^{i}}
		\genfrac{\{}{\}}{0pt}{}{i-1}{j-1}%
	s\left(  j,p\right)  \right).
	\end{align*}
	From
	\[%
	{\displaystyle\sum\limits_{j=1}^{i}}
	\genfrac{\{}{\}}{0pt}{}{i-1}{j-1}%
	s\left(  j,p\right)  =\left(  -1\right)  ^{i-p}\binom{i-1}{p-1}
	\]
	and%
	\[%
	{\displaystyle\sum\limits_{i=1}^{n}}
	\left(  -1\right)  ^{i-p}i!%
	\genfrac{\{}{\}}{0pt}{}{n}{i}%
	\binom{i-1}{p-1}=\left(  -1\right)  ^{n-p}p!%
	\genfrac{\{}{\}}{0pt}{}{n}{p},
	\]
	we get the desired result.
\end{proof}
Proceeding in the same manner as above, one obtains:
\begin{theorem}
	The $\mathcal{B}_{n,m}^{\left(  k\right)  }\left(
	a,q,L\right)  $ is related to the $\widehat{\mathcal{C}}_{n,m}^{\left(  k\right)  }\left(  a,q,L\right)  $
	by the explicit formula%
	\[
	\mathcal{B}_{n,m}^{\left(  k\right)  }\left(  a,q,L\right)  =\left(
	-1\right)  ^{n}%
	{\displaystyle\sum\limits_{j=1}^{n}}
	{\displaystyle\sum\limits_{i=1}^{n}}
	q^{n-j}i!
		\genfrac{\{}{\}}{0pt}{}{n}{i}%
	\genfrac{\{}{\}}{0pt}{}{i}{j}%
	\widehat{\mathcal{C}}_{j,m}^{\left(  k\right)  }\left(  a,q,L\right), \ \ \ (n\geq 1).
	\]
\end{theorem}
\begin{theorem}
	The following formulas hold true:
	\begin{align*}
	\mathcal{C}_{n,m}^{\left(  k\right)  }\left(  a,q,L\right)   &  =%
	{\displaystyle\sum\limits_{j=1}^{n}}
	{\displaystyle\sum\limits_{i=1}^{n}}q^{n-j}
	\frac{\left(  -1\right)  ^{i-j}}{i!}s\left(  n,i\right)  s\left(  i,j\right)
	\mathcal{B}_{j,m}^{\left(  k\right)  }\left(  a,q,L\right), \ \ \ (n\geq 1), \\
	\widehat{\mathcal{C}}_{n,m}^{\left(  k\right)  }\left(  a,q,L\right)   &  =%
		{\displaystyle\sum\limits_{j=1}^{n}}
		{\displaystyle\sum\limits_{i=1}^{n}}\left(  -1\right)  ^{j}q^{n-j}
	\frac{1}{i!}s\left(  n,i\right)  s\left(  i,j\right)  \mathcal{B}%
	_{j,m}^{\left(  k\right)  }\left(  a,q,L\right), \ \ \ (n\geq 1).
	\end{align*}
\end{theorem}
\begin{proof}
	Since%
	\[%
		{\displaystyle\sum\limits_{j=1}^{i}}
	s\left(  i,j\right)
		\genfrac{\{}{\}}{0pt}{}{j}{p}%
	=\left\{
	\begin{array}
	[c]{c}%
	1,\text{ \ \ \ \ \ \ \ \ \ }(i=p)\\
	0,\text{ \ \ }(otherwise)
	\end{array}
	\right.
	\]
		\begin{align*}
	RHS  & =%
	{\displaystyle\sum\limits_{j=1}^{n}}
		{\displaystyle\sum\limits_{i=1}^{n}}q^{n-j}
	\frac{\left(  -1\right)  ^{i-j}}{i!}s\left(  n,i\right)  s\left(  i,j\right)
	\left(  \frac{\left(  a+m\right)  ^{k}}{a^{k}}%
		{\displaystyle\sum\limits_{p=0}^{j}}
	\frac{p!\left(  -q\right)  ^{j-p}l^{p+a}}{\left(  a+m+p\right)  ^{k}}%
		\genfrac{\{}{\}}{0pt}{}{j}{p}%
	\right)  \\
	& =\frac{\left(  a+m\right)  ^{k}}{a^{k}}%
		{\displaystyle\sum\limits_{p=0}^{n}}
	\frac{p!q^{n-p}l^{p+a}}{\left(  a+m+p\right)  ^{k}}%
	{\displaystyle\sum\limits_{i=1}^{n}}
	\frac{\left(  -1\right)  ^{i-p}}{i!}s\left(  n,i\right)
		{\displaystyle\sum\limits_{j=1}^{i}}
	s\left(  i,j\right)
		\genfrac{\{}{\}}{0pt}{}{j}{p}%
	\\
	& =\frac{\left(  a+m\right)  ^{k}}{a^{k}}%
		{\displaystyle\sum\limits_{p=0}^{n}}
	\frac{q^{n-p}l^{p+a}}{\left(  a+m+p\right)  ^{k}}s\left(  n,p\right)  \\
	& =\mathcal{C}_{n,m}^{\left(  k\right)  }\left(  a,q,L\right).
	\end{align*}
	The second identity can be shown similarly.
\end{proof}
\section{The generalized $m$-poly-Cauchy polynomials}
\subsection{The generalized $m$-poly-Cauchy polynomials of the first kind}
Let $n\geq0,m\geq0$ and $k\geq1$ be integers, and let $a,q$ and
$l_{1},...,l_{k\text{ }}$be non-zero real numbers, with $L=\left(
l_{1},...,l_{k}\right)  $ and $l=l_{1},...,l_{k\text{ }}.$ By $\left(
x\right)  _{q,n}$ we denote the generalized rising factorial given by
  $\left(  x\right)  _{q,n}=x\left(
x+q\right)  \cdots\left(  x+q(n-1)\right)$ with $\left(  x\right)  _{q,0}=1$.

We define $\mathcal{C}_{n,m}^{\left(  k\right)  }\left(  a,q,L;x\right)  $ as follows
\begin{equation}
\frac{l^{a}\left(  a+m\right)  ^{k}}{a^{k}\left(  1+qz\right)  ^{\frac{x}{q}}%
}\operatorname*{Lif}\nolimits_{k}\left(  \frac{l\ln\left(  1+qz\right)  }%
{q};a+m\right)  =%
{\displaystyle\sum\limits_{n\geq0}}
\mathcal{C}_{n,m}^{\left(  k\right)  }\left(  a,q,L;x\right)  \frac{z^{n}}%
{n!},\label{GG1}%
\end{equation}
which we call the generalized $m$-poly-Cauchy polynomials of the first kind
and, we have the following expression for $\mathcal{C}_{n,m}^{\left(  k\right)  }\left(
a,q,L;x\right)  $.
\begin{theorem}\label{thm12}
One has
\begin{equation}
\mathcal{C}_{n,m}^{\left(  k\right)  }\left(  a,q,L;x\right)  =%
{\displaystyle\sum\limits_{j=0}^{n}}
\left(  -1\right)  ^{n-j}\binom{n}{j}\mathcal{C}_{j,m}^{\left(  k\right)
}\left(  a,q,L\right)  \left(  x\right)  _{q,n-j}.\label{GGF}%
\end{equation}
\end{theorem}
\begin{proof}
From (\ref{GG1}) and
\[%
{\displaystyle\sum\limits_{n\geq0}}
\left(  -1\right)  ^{n}\left(  x\right)  _{q,n}\frac{z^{n}}{n!}=\frac
{1}{\left(  1+qz\right)  ^{\frac{x}{q}}},
\]
we find
\begin{align*}%
{\displaystyle\sum\limits_{n\geq0}}
\mathcal{C}_{n,m}^{\left(  k\right)  }\left(  a,q,L;x\right)  \frac{z^{n}}{n!}
&  =\left(
{\displaystyle\sum\limits_{n\geq0}}
\left(  -1\right)  ^{n}\left(  x\right)  _{q,n}\frac{z^{n}}{n!}\right)
\left(
{\displaystyle\sum\limits_{n\geq0}}
\mathcal{C}_{n,m}^{\left(  k\right)  }\left(  a,q,L\right)  \frac{z^{n}}
{n!}\right) \\
&  =%
{\displaystyle\sum\limits_{n\geq0}}
\left(
{\displaystyle\sum\limits_{j=0}^{n}}
\left(  -1\right)  ^{n-j}\binom{n}{j}\mathcal{C}_{j,m}^{\left(  k\right)
}\left(  a,q,L\right)  \left(  x\right)  _{q,n-j}\right)  \frac{z^{n}}{n!}.
\end{align*}
The theorem follows from coefficient comparison.
\end{proof}
\begin{theorem} \label{thm13}
One has%
\[
\mathcal{C}_{n,m}^{\left(  k\right)  }\left(  a,q,L;x\right)  =\frac{\left(
a+m\right)  ^{k}}{a^{k}}\sum_{i=0}^{n}\mathcal{T}_{n}^{i}\left(  \frac{x}%
{q}\right)  \frac{q^{n-i}l^{a+i}}{\left(  a+i+m\right)  ^{k}},
\]
wherein $\mathcal{T}_{n}^{i}\left(  x\right)  $ is the weighted Stirling
numbers of the first kind \cite{Car1,Car2}, given by
\[
\frac{1}{i!}\frac{\left(  \ln\left(  1+z\right)  \right)  ^{i}}{\left(
1+z\right)  ^{x}}=%
{\displaystyle\sum\limits_{n=i}^{\infty}}
\mathcal{T}_{n}^{i}\left(  x\right)  \frac{z^{n}}{n!}
\]
or explicitly by%
\begin{equation}
\mathcal{T}_{n}^{i}\left(  x\right)  =%
{\displaystyle\sum\limits_{j=0}^{n}}
\left(  -1\right)  ^{n-j}\binom{n}{j}s\left(  j,i\right)  \left(  x\right)
_{1,n-j}.\label{GG2}%
\end{equation}
\end{theorem}
\begin{proof}
From $\left(  \ref{GGF}\right) $ and Theorem \ref{GBF1}, we can deduce%
\begin{align*}
\mathcal{C}_{n,m}^{\left(  k\right)  }\left(  a,q,L;x\right)   & =%
{\displaystyle\sum\limits_{j=0}^{n}}
\left(  -1\right)  ^{n-j}\binom{n}{j}\mathcal{C}_{j,m}^{\left(  k\right)
}\left(  a,q,L\right)  \left(  x\right)  _{q,n-j}\\
& =%
{\displaystyle\sum\limits_{j=0}^{n}}
\left(  -1\right)  ^{n-j}\binom{n}{j}\frac{\left(  a+m\right)  ^{k}}{a^{k}}%
{\displaystyle\sum\limits_{i=0}^{j}}
s\left(  j,i\right)  \frac{q^{j-i}l^{a+i}}{\left(  a+i+m\right)  ^{k}}\left(
x\right)  _{q,n-j}\\
& =\frac{\left(  a+m\right)  ^{k}}{a^{k}}%
{\displaystyle\sum\limits_{i=0}^{n}}
\frac{q^{n-i}l^{a+i}}{\left(  a+i+m\right)  ^{k}}%
{\displaystyle\sum\limits_{j=0}^{n}}
\left(  -1\right)  ^{n-j}\binom{n}{j}s\left(  j,i\right)  \left(  \frac{x}%
{q}\right)  _{n-j}%
\end{align*}
and the theorem follows from $\left(  \ref{GG2}\right)  $.
\end{proof}
\begin{theorem}\label{thm14}
The following identity is true
\[
\sum_{i=0}^{n}q^{n-i}\mathcal{S}_{n}^{i}\left(  \frac{x}{q}\right)
\mathcal{C}_{i,m}^{\left(  k\right)  }\left(  a,q,L;x\right)  =\frac{\left(
a+m\right)  ^{k}l^{n+a}}{a^{k}\left(  a+n+m\right)  ^{k}},
\]
where $\mathcal{S}_{n}^{i}\left(  x\right)  $ is the weighted Stirling
numbers of the second kind \cite{Car1,Car2} given by
\begin{equation}
{\displaystyle\sum\limits_{n=i}^{\infty}}
\mathcal{S}_{n}^{i}\left(  x\right)  \frac{z^{n}}{n!}=\frac{1}{i!}%
e^{xz}\left(  e^{z}-1\right)  ^{i}\label{GENS},
\end{equation}
or explicitly by%
\[
\mathcal{S}_{n}^{i}\left(  x\right)  =\frac{1}{i!}%
{\displaystyle\sum\limits_{j=0}^{i}}
\binom{i}{j}\left(  -1\right)  ^{i-j}\left(  x+j\right)  ^{n}.
\]
\end{theorem}
\begin{proof}
By using the weighted Stirling transform \cite{Car1}%
\[
\alpha_{n}=%
{\displaystyle\sum\limits_{i=0}^{n}}
\mathcal{S}_{n}^{i}\left(  x\right)  \beta_{i}\Leftrightarrow\beta_{n}=%
{\displaystyle\sum\limits_{i=0}^{n}}
\mathcal{T}_{n}^{i}\left(  x\right)  \alpha_{i},
\]
and Theorem \ref{thm13}, we conclude the proof.
\end{proof}
\subsection{The generalized $m$-poly-Cauchy polynomials of the second kind}
Similarly, we define the generalized $m$-poly-Cauchy polynomials of the second
kind $\widehat{\mathcal{C}}_{n,m}^{\left(  k\right)  }\left(  a,q,L;x\right)
$ by%
\[
\frac{l^{a}\left(  a+m\right)  ^{k}}{a^{k}\left(  1+qz\right)  ^{\frac{x}{q}}%
}\operatorname*{Lif}\nolimits_{k}\left(  -\frac{l\ln\left(  1+qz\right)  }%
{q};a+m\right)  =%
{\displaystyle\sum\limits_{n\geq0}}
\widehat{\mathcal{C}}_{n,m}^{\left(  k\right)  }\left(  a,q,L;x\right)
\frac{z^{n}}{n!}.
\]
\begin{theorem} \label{thm15}
We have
\begin{enumerate}
\item The $\widehat{\mathcal{C}}_{n,m}^{\left(  k\right)  }\left(
a,q,L;x\right)  $ may be
expressed as
\begin{equation}
\widehat{\mathcal{C}}_{n,m}^{\left(  k\right)  }\left(  a,q,L;x\right)
=\frac{\left(  a+m\right)  ^{k}}{a^{k}}{\displaystyle\sum\limits_{i=0}^{n}%
}q^{n-i}\mathcal{T}_{n}^{i}\left(  \frac{x}{q}\right)  \frac{\left(
-1\right)  ^{i}l^{a+i}}{\left(  a+i+m\right)  ^{k}}. \label{ws}
\end{equation}
\item The $\widehat{\mathcal{C}}_{n,m}^{\left(  k\right)  }\left(
a,q,L;x\right)  $ may be also expressed as%
\[
\widehat{\mathcal{C}}_{n,m}^{\left(  k\right)  }\left(  a,q,L;x\right)  =%
{\displaystyle\sum\limits_{j=0}^{n}}
\left(  -1\right)  ^{n-j}\binom{n}{j}\widehat{\mathcal{C}}_{j,m}^{\left(
k\right)  }\left(  a,q,L\right)  \left(  x\right)  _{q,n-j}.
\]

\end{enumerate}
\end{theorem}
Next, we shall prove some relations between the aforementioned two kinds of the generalized $m$-poly-Cauchy polynomials.
\begin{theorem} \label{thm16}
For $n\geq1,$
\begin{align}
\left(  -1\right)  ^{n}\frac{\mathcal{C}_{n,m}^{\left(  k\right)  }\left(
a,q,L;-x\right)  }{n!} &  =%
{\displaystyle\sum\limits_{i=1}^{n}}
\binom{n-1}{i-1}q^{n-i}\frac{\widehat{\mathcal{C}}_{i,m}^{\left(  k\right)
}\left(  a,q,L;x\right)  }{i!},\label{zR1}\\
\left(  -1\right)  ^{n}\frac{\widehat{\mathcal{C}}_{n,m}^{\left(  k\right)
}\left(  a,q,L;-x\right)  }{n!} &  =%
{\displaystyle\sum\limits_{i=1}^{n}}
\binom{n-1}{i-1}q^{n-i}\frac{\mathcal{C}_{i,m}^{\left(  k\right)  }\left(
a,q,L;x\right)  }{i!}.\label{zR2}%
\end{align}
\end{theorem}
\begin{proof}
Applying generating function formula, we can have
\begin{align*}%
{\displaystyle\sum\limits_{n\geq0}}
\left(  -1\right)  ^{n}\frac{\mathcal{C}_{n,m}^{\left(  k\right)  }\left(
a,q,L;-x\right)  }{n!}z^{n}  & =%
{\displaystyle\sum\limits_{n\geq0}}
\mathcal{C}_{n,m}^{\left(  k\right)  }\left(  a,q,L;-x\right)  \frac{\left(
-z\right)  ^{n}}{n!}\\
& =\left(  1-qz\right)  ^{\frac{x}{q}}\frac{l^{a}\left(  a+m\right)  ^{k}%
}{a^{k}}\operatorname*{Lif}\nolimits_{k}\left(  \frac{l\ln\left(  1-qz\right)
}{q};a+m\right)  .
\end{align*}
On the other hand,
\[%
{\displaystyle\sum\limits_{n\geq0}}
\left(
{\displaystyle\sum\limits_{i=1}^{n}}
\binom{n-1}{i-1}q^{n-i}\frac{\widehat{\mathcal{C}}_{i,m}^{\left(  k\right)
}\left(  a,q,L;x\right)  }{i!}\right)  z^{n}=%
{\displaystyle\sum\limits_{i\geq1}}
\frac{\widehat{\mathcal{C}}_{i,m}^{\left(  k\right)  }\left(  a,q,L;x\right)
}{i!q^{i}}%
{\displaystyle\sum\limits_{n\geq i}}
\binom{n-1}{i-1}\left(  qz\right)  ^{n}.
\]
Since
\[%
{\displaystyle\sum\limits_{n\geq i}}
\binom{n-1}{i-1}z^{n}=\frac{z^{i}}{\left(  1-z\right)  ^{i}},
\]
we get
\begin{align*}%
{\displaystyle\sum\limits_{n\geq0}}
\left(
{\displaystyle\sum\limits_{i=1}^{n}}
q^{n-i}\binom{n-1}{i-1}\frac{\widehat{\mathcal{C}}_{i,m}^{\left(  k\right)
}\left(  a,q,L;x\right)  }{i!}\right)  z^{n} &  =%
{\displaystyle\sum\limits_{i\geq1}}
\frac{\widehat{\mathcal{C}}_{i,m}^{\left(  k\right)  }\left(  a,q,L;x\right)
}{i!}\frac{z^{i}}{\left(  1-qz\right)  ^{i}}\\
= &  \left(  1-qz\right)  ^{\frac{x}{q}}\frac{l^{a}\left(  a+m\right)  ^{k}%
}{a^{k}}\operatorname*{Lif}\nolimits_{k}\left(  \frac{l\ln\left(  1-qz\right)
}{q};a+m\right).
\end{align*}
Thus, the proof of (\ref{zR1}) is completed. Using the same method we can prove
(\ref{zR2}).
\end{proof}
%%%%%%%%%%%%%%%%%%%%%%%%%%%%%%%%%%%%%%%%%%%%%%%%%%%%%%%%%%%%%%%%%%%%%%%%%%%%%%%%%%%%%%%%%%%%%%%%%%%%%%%%%%%%%%%%%%%%%%%%%%%%
%%%%%%%%% Bernoulli polynomials
%%%%%%%%%%%%%%%%%%%%%%%%%%%%%%%%%%%%%%%%%%%%%%%%%%%%%%%%%%%%%%%%%%%%%%%%%%%%%%%%%%%%%%%%%%%%%%%%%%%%%%%%%%%%%%%%%%%%%%%%%
\section{The generalized $m-$poly-Bernoulli polynomials}
\begin{definition}
Let $n\geq0,m\geq0$ and $k\geq1$ be integers, and let $a,q$ and $l_{1}%
,...,l_{k\text{ }}$ non-zero real numbers, with $L=\left(  l_{1},...,l_{k\text{
}}\right)  $ and $l=\prod_{i=1}^{k}l_{i}$. We define the generalized $m$-poly-Bernoulli
polynomials $\mathcal{B}_{n,m}^{(k)}\left(  a,q,L;x\right)
$ by the formula
\begin{equation}
\frac{ql^{a-1}\left(  a+m\right)  ^{k}}{a^{k}}\frac{\operatorname*{Li}_{k}\left(  \frac{l}%
{q}\left(  1-e^{-qz}\right)  ;a+m-1\right)  }{e^{qzx}\left(  1-e^{-qz}\right)
}=\sum_{n\geq0}\mathcal{B}_{n,m}^{(k)}\left(  a,q,L;x\right)  \frac{z^{n}}%
{n!},\label{For4}
\end{equation}
with
\[
\mathcal{B}_{n,m}^{(k)}\left(  a,q,L\right):=\mathcal{B}_{n,m}^{(k)}\left(  a,q,L;0\right).
\]
\end{definition}
The following expression will prove that the polynomials $\mathcal{B}_{n,m}^{(k)}\left(  a,q,L;x\right)
$ can be described in terms of $\mathcal{S}_{n}^{i}\left(  x\right) $:
\begin{theorem}\label{th17}
 For $\mathcal{B}_{n,m}^{(k)}\left(  a,q,L;x\right),  $ we have
\begin{equation}
\mathcal{B}_{n,m}^{(k)}\left(  a,q,L;x\right)  =\frac{\left(  a+m\right)
^{k}}{a^{k}}\sum_{i=0}^{n}\frac{i!\left(  -q\right)  ^{n-i}l^{i+a}}{\left(
a+m+i\right)  ^{k}}\mathcal{S}_{n}^{i}\left(  x\right)  \label{For5}.
\end{equation}
\end{theorem}
\begin{proof}
Equations (\ref{For4}) and (\ref{GENS}) give
\begin{align*}
\sum_{n\geq0}\mathcal{B}_{n,m}^{(k)}\left(  a,q,L;x\right)  \frac{z^{n}}{n!}
& =\frac{ql^{a-1}\left(  a+m\right)  ^{k}}{a^{k}e^{qzx}\left(  1-e^{-qz}%
\right)  }\sum_{i\geq1}\frac{\left(  \frac{l}{q}\left(  1-e^{-qz}\right)
\right)  ^{i}}{\left(  i+a+m-1\right)  ^{k}}\\
& =\frac{l^{a}\left(  a+m\right)  ^{k}}{a^{k}e^{qzx}}\sum_{i\geq0}%
\frac{\left(  \frac{l}{q}\left(  1-e^{-qz}\right)  \right)  ^{i}}{\left(
i+a+m\right)  ^{k}}\\
& =\frac{\left(  a+m\right)  ^{k}}{a^{k}}\sum_{i\geq0}\frac{\left(  -1\right)
^{i}i!l^{i+a}}{q^{i}\left(  i+a+m\right)  ^{k}}\sum_{n\geq i}\mathcal{S}_{n}^{i}\left(
x\right)  \frac{\left(  -qz\right)  ^{n}}{n!}\\
& =\sum_{n\geq0}\left(  \frac{\left(  a+m\right)  ^{k}}{a^{k}}\sum_{i=0}%
^{n}\frac{\left(  -q\right)  ^{n-i}i!l^{i+a}}{\left(  i+a+m\right)  ^{k}}%
\mathcal{S}_{n}^{i}\left(  x\right)  \right)  \frac{z^{n}}{n!}.
\end{align*}
By equating two power series, we get the result.
\end{proof}
In particular,
\[
\mathcal{B}_{n,m}^{(k)}\left(  a,q,L;r\right)  =\frac{\left(  a+m\right)
^{k}}{a^{k}}\sum_{i=0}^{n}\frac{i!\left(  -q\right)  ^{n-i}l^{i+a}}{\left(
a+m+i\right)  ^{k}}%
\genfrac{\{}{\}}{0pt}{}{n+r}{i+r}%
_{r}
\]
and%
\[
\mathcal{B}_{n,m}^{(k)}\left(  a,q,L;\frac{r}{h}\right)  =\frac{\left(
a+m\right)  ^{k}}{a^{k}}\sum_{i=0}^{n}\frac{1}{h^{n-i}}\frac{i!\left(
-q\right)  ^{n-i}l^{i+a}}{\left(  a+m+i\right)  ^{k}}W_{h,r}\left(
n,i\right),
\]
where $\mathcal{S}_{n}^{i}\left(  x\right)  $ are related to $%
\genfrac{\{}{\}}{0pt}{}{n}{i}%
_{r}$ and $r$-Whitney numbers of the second kind $W_{h,r}\left(  n,i\right)  $
by%
\[
\mathcal{S}_{n}^{i}\left(  r\right)  =%
\genfrac{\{}{\}}{0pt}{}{n+r}{i+r}%
_{r}\ \text{and }\ \mathcal{S}_{n}^{i}\left(  \frac{r}{h}\right)  =\frac{1}{h^{n-i}}%
W_{h,r}\left(  n,i\right).
\]
\begin{theorem}
We have
\begin{equation}
\sum_{i=0}^{n}\mathcal{T}_{n}^{i}\left(  x\right)  (-q)^{n-i}\mathcal{B}_{i,m}%
^{(k)}\left(  a,q,L;x\right)  =\frac{\left(  a+m\right)  ^{k}n!l^{n+a}}%
{a^{k}\left(  n+a+m\right)  ^{k}}\label{For6}.
\end{equation}
\end{theorem}
\begin{proof}
By using the weighted Stirling transform and Theorem \ref{th17}, we get the result.
\end{proof}
\begin{theorem}
An explicit expression for $\mathcal{B}_{n,m}^{(k)}\left(
a,q,L;x\right)  $ is given by
\begin{equation}
\mathcal{B}_{n,m}^{(k)}\left(  a,q,L;x\right)  =\sum_{i=0}^{n}\binom{n}%
{i}\left(  -q\right)  ^{n-i}\mathcal{B}_{i,m}^{(k)}\left(  a,q,L\right)
\text{ }x^{n-i}\label{For7}.
\end{equation}
\end{theorem}
\begin{proof}
\eqref{For4} can be written as:
\begin{align*}
\sum_{n\geq0}\mathcal{B}_{n,m}^{(k)}\left(  a,q,L;x\right)  \frac{z^{n}}{n!}
& =\frac{ql^{a-1}\left(  a+m\right)  ^{k}}{a^{k}}\frac{\operatorname*{Li}_{k}\left(  \frac
{l}{q}\left(  1-e^{-qz}\right)  ;a+m-1\right)  }{e^{qzx}\left(  1-e^{-qz}%
\right)  }\\
& =e^{-qzx}\frac{ql^{a-1}\left(  a+m\right)  ^{k}}{a^{k}}\frac{\operatorname*{Li}_{k}\left(
\frac{l}{q}\left(  1-e^{-qz}\right)  ;a+m-1\right)  }{\left(  1-e^{-qz}%
\right)  }\\
& =e^{-qzx}\sum_{n\geq0}\mathcal{B}_{n,m}^{(k)}\left(  a,q,L\right)
\frac{z^{n}}{n!}\\
& =\sum_{n\geq0}\left(  \sum_{i=0}^{n}\binom{n}{i}\left(  -q\right)
^{n-i}\mathcal{B}_{i,m}^{(k)}\left(  a,q,L\right)  \text{ }x^{n-i}\right)
\frac{z^{n}}{n!},
\end{align*}
which finishes the proof.
\end{proof}
\begin{theorem}
 $\mathcal{B}_{n,m}^{(-k)}\left(
a,q,L;x\right)  $ has the double exponential generating function
\[
\sum_{n\geq0}\sum_{k\geq0}\mathcal{B}_{n,m}^{(-k)}\left(  a,q,L;x\right)
\frac{x^{n}}{n!}\frac{y^{k}}{k!}=\frac{l^{a}e^{ay-qzx}}{1-\frac{l}{q}\left(
1-e^{-qz}\right)  e^{\frac{ay}{a+m}}}.
\]
\end{theorem}
\begin{proof}
From (\ref{For4}), we have
\begin{align*}
\sum_{n\geq0}\sum_{k\geq0}\mathcal{B}_{n,m}^{(-k)}\left(  a,q,L;x\right)
\frac{x^{n}}{n!}\frac{y^{k}}{k!}  & =\sum_{k\geq0}\left(  \frac{e^{-qzx}%
ql^{a-1}a^{k}}{\left(  a+m\right)  ^{k}}\sum_{p\geq0}\frac{\left(  \frac{l}%
{q}\left(  1-e^{-qz}\right)  \right)  ^{p}\left(  a+m+p-1\right)  ^{k}%
}{\left(  1-e^{-qz}\right)  }\right)  \frac{y^{k}}{k!}\\
& =l^{a}e^{-qzx}\sum_{p\geq1}\left(  \frac{l}{q}\left(  1-e^{-qz}\right)
\right)  ^{p-1}\sum_{k\geq0}\left(  \frac{a\left(  a+m+p-1\right)  y}{\left(
a+m\right)  }\right)  ^{k}\frac{1}{k!}\\
& =l^{a}e^{ay}e^{-qzx}\sum_{p\geq1}\left(  \frac{l}{q}\left(  1-e^{-qz}%
\right)  e^{\frac{ay}{a+m}}\right)  ^{p-1}\\
& =\frac{l^{a}e^{ay-qzx}}{1-\frac{l}{q}\left(  1-e^{-qz}\right)  e^{\frac
{ay}{a+m}}}.
\end{align*}
\end{proof}
Next, we state some connections between the generalized
$m$-poly-Bernoulli polynomials and the generalized $m$-poly-Cauchy polynomials.
\begin{theorem}\label{thm200}
For $n\geq0,$
\begin{align}
\mathcal{B}_{n,m}^{(k)}\left(  a,q,L;x\right)   & =\sum_{i=0}^{n}i!\left(
-q\right)  ^{n-i}\mathcal{S}_{n}^{i}\left(  x\right)  \sum_{j=0}^{n}q^{i-j}\mathcal{S}_{i}%
^{j}\left(  \frac{x}{q}\right)  \mathcal{C}_{j,m}^{(k)}\left(  a,q,L;x\right)\label{For9}. \\
\mathcal{B}_{n,m}^{(k)}\left(  a,q,L;x\right)   & =\left(  -1\right)  ^{n}\sum_{j=0}^{n}q^{n-j}\sum_{i=0}^{n}i!\mathcal{S}_{n}%
^{i}\left(  x\right)  \mathcal{S}_{i}^{j}\left(  \frac{x}{q}\right)  \widehat
{\mathcal{C}}_{j,m}^{(k)}\left(  a,q,L;x\right)  \label{For10}
\end{align}
and
\begin{align}
\mathcal{C}_{n,m}^{(k)}\left(  a,q,L;x\right)   & =\sum_{j=0}^{n}q^{n-j}%
\sum_{i=0}^{n}\frac{\left(  -1\right)  ^{i-j}}{i!}\mathcal{T}_{n}^{i}\left(  \frac{x}%
{q}\right)  \mathcal{T}_{i}^{j}\left(  x\right)  \mathcal{B}_{j,m}^{(k)}\left(
a,q,L;x\right)  \label{For11}. \\
\widehat{\mathcal{C}}_{n,m}^{(k)}\left(  a,q,L;x\right)   & =\sum_{j=0}%
^{n}q^{n-j}\sum_{i=0}^{n}\frac{\left(  -1\right)  ^{j}}{i!}\mathcal{T}_{n}^{i}\left(
\frac{x}{q}\right)  \mathcal{T}_{i}^{j}\left(  x\right)  \mathcal{B}_{j,m}^{(k)}\left(
a,q,L;x\right)  \label{For12}.
\end{align}
 \end{theorem}
\begin{proof}
Theorem \ref{thm14} and (\ref{For5}) give
\begin{align*}
RHS  & =\sum_{i=0}^{n}\left(  -q\right)  ^{n-i}\sum_{j=0}^{n}q^{i-j}%
i!\mathcal{S}_{n}^{i}\left(  x\right)  \mathcal{S}_{i}^{j}\left(  \frac{x}{q}\right)
\mathcal{C}_{j,m}^{(k)}\left(  a,q,L;x\right)  \text{ \ }\\
& =\sum_{i=0}^{n}\left(  -q\right)  ^{n-i}i!\mathcal{S}_{n}^{i}\left(  x\right)
\sum_{j=0}^{i}q^{i-j}\mathcal{S}_{i}^{j}\left(  \frac{x}{q}\right)  \mathcal{C}%
_{j,m}^{(k)}\left(  a,q,L;x\right) \\
& =\sum_{i=0}^{n}\left(  -q\right)  ^{n-i}i!\mathcal{S}_{n}^{i}\left(  x\right)
\frac{\left(  a+m\right)  ^{k}l^{i+a}}{a^{k}\left(  a+i+m\right)  ^{k}}\\
& =\mathcal{B}_{n,m}^{(k)}\left(  a,q,L;x\right).
\end{align*}
This completes the proof of (\ref{For9}). For (\ref{For10}), using the inverse Stirling transform of (\ref{ws}) and the identity (\ref{For5}), we obtain
\begin{align*}
RHS  & =\left(  -1\right)  ^{n}\sum_{j=0}^{n}q^{n-j}\sum_{i=0}^{n}i!\mathcal{S}_{n}%
^{i}\left(  x\right)  \mathcal{S}_{i}^{j}\left(  \frac{x}{q}\right)  \widehat
{\mathcal{C}}_{j,m}^{(k)}\left(  a,q,L;x\right)  \text{ \ }\\
& =\sum_{i=0}^{n}\left(  -1\right)  ^{n}i!\mathcal{S}_{n}^{i}\left(  x\right)  q^{n}%
\sum_{j=0}^{i}q^{-j}\mathcal{S}_{i}^{j}\left(  \frac{x}{q}\right)  \widehat{\mathcal{C}%
}_{j,m}^{(k)}\left(  a,q,L;x\right) \\
& =\sum_{i=0}^{n}\left(  -1\right)  ^{n}i!\mathcal{S}_{n}^{i}\left(  x\right)
q^{n}\frac{\left(  a+m\right)  ^{k}\left(  -1\right)  ^{i}q^{-i}l^{i+a}}%
{a^{k}\left(  a+i+m\right)  ^{k}}\\
& =\sum_{i=0}^{n}\left(  -q\right)  ^{n-i}i!\mathcal{S}_{n}^{i}\left(  x\right)
\frac{\left(  a+m\right)  ^{k}l^{i+a}}{a^{k}\left(  a+i+m\right)  ^{k}}\\
& =\mathcal{B}_{n,m}^{(k)}\left(  a,q,L;x\right).
\end{align*}
This completes the proof of (\ref{For10}). For (\ref{For11}), by (\ref{For6}) and Theorem \ref{thm13}, we have%
\begin{align*}
RHS  & =\sum_{j=0}^{n}q^{n-j}\sum_{i=0}^{n}\frac{\left(  -1\right)  ^{i-j}%
}{i!}\mathcal{T}_{n}^{i}\left(  \frac{x}{q}\right)  \mathcal{T}_{i}^{j}\left(  x\right)
\mathcal{B}_{j,m}^{(k)}\left(  a,q,L;x\right)  \text{ }\\
& =\sum_{i=0}^{n}\frac{q^{n-i}}{i!}\mathcal{T}_{n}^{i}\left(  \frac{x}{q}\right)
\sum_{j=0}^{i}\left(  -q\right)  ^{i-j}\mathcal{T}_{i}^{j}\left(  x\right)
\mathcal{B}_{j,m}^{(k)}\left(  a,q,L;x\right) \\
& =\sum_{i=0}^{n}\frac{q^{n-i}}{i!}\mathcal{T}_{n}^{i}\left(  \frac{x}{q}\right)
\frac{\left(  a+m\right)  ^{k}i!l^{i+a}}{a^{k}\left(  a+i+m\right)  ^{k}}\\
& =\mathcal{C}_{n,m}^{(k)}\left(  a,q,L;x\right).
\end{align*}
This completes the proof of (\ref{For11}). The last identity (\ref{For12}) can be shown similarly.
\end{proof}
%%%%%%%%%%%%%%%%%%%%%%%%%%%%%%%%%%%%%%%%%%%%%%%%%%%%%%%%%%%%%%%%%
%\textbf{A three-term recurrence relation, for calculating the generalized}$\mathbf{m}$\textbf{-poly-Bernoulli polynomials}
%%%%%%%%%%%%%%%%%%%%%%%%%%%%%%%%%%%%%%%%%%%%%%%%%%%%%%%%%%%%%%%%%%%%

Now, we generalize (\ref{TGQ1}) to the polynomial case. For this, we define the  polynomials
$H_{n,p}\left(  m;x\right)  $ by means of the following expression%
\begin{equation}
H_{n,p}\left(  m;x\right)  :=H_{n,m,p}^{\left(  k\right)  }\left(
a,q,L;x\right)  =\sum_{i=0}^{n}\binom{n}{i}\left(  -q\right)  ^{n-i}%
H_{i,p}\left(  m\right)  \text{ }x^{n-i}\label{For13},
\end{equation}
with
\[
H_{0,p}\left(  m;x\right)  =\frac{l^{a}}{a^{k}}
\]
and
\[
H_{n,0}\left(
m;x\right)  =\mathcal{B}_{n,m}^{(k)}\left(  a,q,L;x\right)  .
\]
Here we can present another explicit expression for $H_{n,p}\left(  m;x\right) $
\[
H_{n,p}\left(  m;x\right)  =\frac{\left(  p+a+m\right)  ^{k}}{p!a^{k}}%
\sum_{i=0}^{n}  \frac{\left(  i+p\right)  !\left(
-q\right)  ^{n-i}l^{i+a}}{\left(  a+m+i+p\right)  ^{k}}\mathcal{S}_{n}^{i}\left(  x+p\right).
\]
\begin{theorem}
The $H_{n,p}\left(  m;x\right)  $ satisfies the following recurrence
\begin{equation}
\text{\ }H_{n,p+1}\left(  m;x\right)  =\frac{l\left(  p+1\right)  \left(
p+a+m\right)  ^{k}}{\left(  p+a+m+1\right)  ^{k}}H_{n+1,p}\left(  m;x\right)
-q \left(  x+p\right)  H_{n,p}\left(  m;x\right)  \label{For14},
\end{equation}
with $H_{0,p}\left(  m;x\right)  =\frac{l^{a}%
}{a^{k}}.$
\end{theorem}
\begin{proof}
By (\ref{For13}) and (\ref{TGQ1}), we can write%
\begin{align*}
x\frac{d}{dx}H_{n,p}\left(  m;x\right)   & =nH_{n,p}\left(  m;x\right)
-n\sum_{i=1}^{n}\binom{n}{i}\frac{i}{n}\left(  -q\right)  ^{n-i}H_{i,p}\left(
m\right)  x^{n-i}\\
& =nH_{n,p}\left(  m;x\right)  -n\sum_{i=1}^{n}\binom{n-1}{i-1}\left(
-q\right)  ^{n-i}H_{i,p}\left(  m\right)  x^{n-i}\\
& =nH_{n,p}\left(  m;x\right)  -n\sum_{i=0}^{n-1}\binom{n-1}{i}\left(
-q\right)  ^{n-i-1}H_{i+1,p}\left(  m\right)  x^{n-i-1}\\
& =nH_{n,p}\left(  m;x\right) -n\frac{l\left(  p+1\right)  \left(  p+a+m\right)  ^{k}}{\left(
p+a+m+1\right)  ^{k}}\sum_{i=0}^{n-1}\binom{n-1}{i}\left(  -q\right)
^{n-i-1}H_{i,p+1}\left(  m\right)  x^{n-i-1}\\
& +npq\sum_{i=0}^{n-1}\binom{n-1}{i}\left(  -q\right)  ^{n-i-1}H_{i,p}\left(
m\right)  x^{n-i-1}.
\end{align*}
Thus%
\begin{align*}
-xnqH_{n-1,p}\left(  m;x\right)   & =nH_{n,p}\left(  m;x\right)
-n\frac{l\left(  p+1\right)  \left(  p+a+m\right)  ^{k}}{\left(
p+a+m+1\right)  ^{k}}H_{n-1,p+1}\left(  m;x\right) \\
& +npqH_{n-1,p}\left(  m;x\right).
\end{align*}
This evidently equivalent to (\ref{For14}).
\end{proof}

%\textbf{Acknowledgements}(If any)

\end{document}